\documentclass{conm-p-l}
\usepackage[mathscr]{eucal}
\usepackage{amsfonts, amsmath, amsthm, amssymb, latexsym}
\newtheorem{theorem}{Theorem}[section]
\newtheorem{lemma}[theorem]{Lemma}
\newtheorem{prop}[theorem]{Proposition}
\newtheorem{cor}[theorem]{Corollary}
\theoremstyle{definition}
\newtheorem{defi}[theorem]{Definition}
\newtheorem{example}[theorem]{Example}
\theoremstyle{remark}
\newtheorem{remark}[theorem]{Remark}

\numberwithin{equation}{section}

\newcommand{\la}{\label}
\newcommand{\GL}{\mathrm{GL}}
\newcommand{\Ker}{\mathrm{Ker}}
\newcommand{\Res}{\mathrm{Res}}
\newcommand{\Spec}{\mathrm{Spec}}
\newcommand{\Mod}{\mathrm{Mod}}

\newcommand{\End}{\mathrm{End}}

\newcommand{\U}{\mathcal{U}}
\def\reg{\mathrm{reg}}
\def\tor{\mathrm{Tor}}
\def\deg{\mathrm{deg}\,}
\newcommand{\into}{\hookrightarrow}

\def\d{d}
\def\C{\mathbb C}

\def\c{\mathbb{C}}
\def\e{\boldsymbol{e}}
\def\A{\mathcal{A}}

\def\k{\mathbb{K}}
\def\Z{\mathbb{Z}}

\def\QQ{\mathbf{Q}}
\def\N{\mathbb{N}}
\def\P{\mathcal{P}}
\def\T{\mathcal{T}}
\def\D{\mathcal{D}}
\def\O{\mathcal{O}}

\def\h{\mathfrak{h}}
\def\g{\mathfrak{g}}
\def\sln{{\mathfrak{s}\mathfrak{l}}_n}

\def\n{\mathfrak{n}}
\def\r{\mathfrak{r}}

\def\grd{\mathrm{gr}}

\def\sl2{{\mathfrak{s}\mathfrak{l}}_2}
\def\vreg{V_{\rm{reg}}}
\def\Hreg{H_{\rm{reg}}}
\def\Mreg{M_{\rm{reg}}}

\def\KZ{\mathtt{KZ}}

\def\kz{\mathrm{kz}}
\def\H{\mathcal{H}}

\def\Reg{\mathrm{Reg}}
\def\mod{\mathrm{mod}}

\begin{document}
\title[Dunkl Operators and Quasi-invariants]{Dunkl Operators and Quasi-invariants of Complex Reflection Groups}
\author{Yuri Berest}
\address{Department of Mathematics,
Cornell University, Ithaca, NY 14853-4201, USA}
\email{berest@math.cornell.edu}
\thanks{The first author was supported in part by NSF Grant DMS 09-01570.}
\author{Peter Samuelson}
\address{Department of Mathematics,
Cornell University, Ithaca, NY 14853-4201, USA}
\email{psamuelson@math.cornell.edu}

\begin{abstract}
In these notes we give an introduction to representation theory of rational 
Cherednik algebras associated to complex reflection groups. We discuss  
applications of this theory in the study of finite-dimensional 
representations of the Hecke algebras and polynomial quasi-invariants 
of these groups.
\end{abstract}
\maketitle

\section{Introduction}
These are lecture notes of a minicourse given by the first author at the summer school on Quantization
at the University of Notre Dame in June 2011. The notes were written up and expanded by the second author who took the liberty of adding a few interesting results and proofs from the literature. In a broad sense, our goal is to give an introduction to representation theory of rational Cherednik algebras and some of its
recent applications. More specifically, we focus on the two concepts featuring in the title (Dunkl operators and quasi-invariants) and explain the relation between them. The course was originally designed for
graduate students and non-experts in representation theory. In these notes, we tried to preserve an informal style,
even at the expense of making imprecise claims and sacrificing rigor.

The interested reader may find more details and proofs in the following references. The original papers on representation theory of the rational Cherednik algebras
are \cite{EG02}, \cite{BEG03, BEG03b}, \cite{DO03} and \cite{GGOR03}; surveys of various aspects of this theory
can be found in \cite{Rou05}, \cite{Et07}, \cite{Gor07} and \cite{Gor10}.
The quasi-invariants of Coxeter (real reflection) groups first appeared in the work of O.~Chalykh and A.~Veselov \cite{CV90, CV93} (see also \cite{VSC93}); the link to the rational Cherednik algebras associated to these groups was established in \cite{BEG03}; various results and applications of Coxeter quasi-invariants can be found in \cite{EG02b}, \cite{FV02}, \cite{GW03}, \cite{GW06}, \cite{BM08}; for a readable survey, we refer to \cite{ES03}.

The notion of a quasi-invariant for a general complex reflection group was introduced
in \cite{BC11}. This last paper extends the results of \cite{BEG03}, unifies the proofs and
gives new applications of quasi-invariants in representation theory and noncommutative algebra.
It is the main reference for the present lectures.

\section{Lecture 1}

\subsection{Historical remarks}
The theory of rational Cherednik algebras was historically motivated by developments in two different areas: integrable systems and multivariable special functions. In each of these areas, the classical
representation theory of semisimple complex Lie algebras played a prominent role. On the integrable systems side, one should single out the work of M.~Olshanetsky and A.~Perelomov \cite{OP83} who found a remarkable
generalization of the Calogero-Moser integrable systems for an arbitrary semisimple Lie algebra. On the special functions side, the story began with the fundamental discovery of Dunkl operators \cite{Dun89} that transformed much of the present day harmonic analysis. In 1991, G.~Heckman \cite{Hec91} noticed a relationship between the two constructions which is naturally explained by the theory of rational Cherednik algebras. The theory itself was developed
by P.~Etingof and V.~Ginzburg in their seminal paper \cite{EG02}. In fact, this last paper introduced a more
general class of algebras (the so-called symplectic reflection algebras) and studied the representation
theory of these algebras at the `quasi-classical' ($t=0$) level\footnote{As yet another precursor of
the theory of Cherednik algebras, one should mention the beautiful paper \cite{Wil98} which examines the link between the classical Calogero-Moser systems and solutions of the (infinite-dimensional) Kadomtsev-Petviashvili integrable system.}. At the `quantum' ($t=1$) level, the representation theory of the rational Cherednik algebras was developed in
\cite{BEG03, BEG03b}, \cite{DO03} and \cite{GGOR03}, following the insightful suggestion by E.~Opdam and R.~Rouquier to model this theory parallel to the theory of universal enveloping algebras of semisimple complex Lie algebras.

In this first lecture, we briefly review the results of \cite{OP83} and \cite{Dun89} in their original setting and  explain the link between these two papers discovered in \cite{Hec91}. Then, after giving a necessary background on complex reflection groups, we introduce the Dunkl operators and sketch the proof of their commutativity following
\cite{DO03}.
In the next lecture, we define the rational Cherednik algebras and show how Heckman's observation can be reinterpreted
in the language of these algebras.

\subsection{Calogero-Moser systems and the Dunkl operators}
\label{heckmansargumentsection}
\subsubsection{The quantum Calogero-Moser system}
Let $ \h = \c^n $, and let $ \h^\reg  $ denote the complement (in $ \h $) of the union of hyperplanes 
$ x_i - x_j=0 $ for  $ 1 \le i \ne j \le n $. Write $ \c[\h^\reg] $ for the ring of regular functions on $ \h^\reg $ (i.e., the rational functions on $ \c^n $ with poles along $ x_i - x_j = 0 $). Consider the following differential operator acting on $\, \c[\h^\reg]\,$
\[
H = \sum_{i=1}^n \left(\frac{\partial}{\partial x_i}\right)^2 \, - \,c(c+1) \sum_{i\not= j} \frac{1}{(x_i-x_j)^2}\ ,
\]
where $ c $ is a complex parameter. This operator is called the quantum Calogero-Moser Hamiltonian: it can be viewed
as the Schr\"odinger operator of the system of $n$ quantum particles on the line with pairwise interaction proportional to $(x_i-x_j)^{-2}$. It turns out that $ H $ is part of a family of $n$ partial differential operators $\, L_j:\, \c[\h^\reg] \to \c[\h^\reg]\,$ of the form
\[
L_j := \sum_{i=1}^n \left(\frac{\partial}{\partial x_i}\right)^j + \textrm{lower order terms} \ ,
\]
satisfying the properties:
\begin{enumerate}
\item $L_j$ is $S_n$-invariant (where $S_n$ is the $n$-th symmetric group acting on $\c^n$ by
permutations of coordinates),
\item $L_j$ is a homogeneous operator of degree $(-j)$,
\item $L_2 = H$,
\item $[L_j,L_k] = 0$, i.e. the $L_j$'s commute.
\end{enumerate}
The last property (commutativity) means that $ \{L_j\}_{1 \leq j \leq n} $ form a quantum integrable system which is called the (quantum) Calogero-Moser system.

\subsubsection{A root system generalization}
\la{opcm}
In \cite{OP83}, the authors constructed a family of commuting differential operators for a Lie algebra
$ \mathfrak g $ that specializes to the Calogero-Moser system when $\g = \sln(\c)$. Precisely,
let $ \mathfrak g $ be a finite-dimensional complex semisimple Lie algebra, $\,\h \subset \g\,$ a Cartan subalgebra
of $\g$, $ \h^* $ its dual vector space, $W$ the corresponding Weyl group, $\,R \subset \h^*\,$ a system of roots of $ \g $, and $ R_+ \subset R $
a choice of positive roots (see, e.g., \cite{Hum78} for definitions of these standard terms).
To each positive root $ \alpha \in R_+ $ we assign a complex number (multiplicity) $ c_\alpha \in \c $
so that two roots in the same orbit of $W$ in $R$ have the same multiplicities; in other words,
we define a $W$-invariant function $\,c:\,R_+ \to \c\,$, which we write as $\, \alpha \mapsto c_\alpha $.
Now, let $\h^\reg$ denote the complement (in $\h$) of the union of reflection hyperplanes of $W$.
The operator $H$ of the previous section generalizes to the differential operator
$\,H_\g:\,\c[\h^\reg] \to \c[\h^\reg]\,$, which is defined by the formula
\[
H_\g := \Delta_\h - \sum_{\alpha \in R_+} \frac{c_\alpha(c_{\alpha} + 1)(\alpha,\alpha)}{(\alpha,x)^2}\ ,
\]
where $ \Delta_\h $ stands for the usual Laplacian on $\h$.

\begin{theorem}[see \cite{OP83}]
\la{OPgen}
For each $\,P \in \c[\h^*]^W\,$, there is a differential operator
\[
L_P = P(\partial_\xi) + \textrm{lower order terms}
\]
acting on $ \c[\h^\reg]\,$, with lower order terms depending on $c$, such that
\begin{enumerate}
\item $L_P$ is $W$-invariant with respect to the natural action of $W$ on $ \h^\reg $
\item $ L_P $ is homogeneous of degree $-\,\deg(P)$,
\item $ L_{\lvert x\rvert^2} = H_\g$,
\item $[L_P,L_Q] = 0$ for all $\,P,Q \in \c[\h^*]^W$.
\end{enumerate}
The assignment $\,P \mapsto L_P\,$ defines an injective algebra homomorphism
$$
\c[\h^*]^W \hookrightarrow \D(\h^\reg)^W \ ,
$$
where $ \D(\h^\reg)^W  $ is the ring of $W$-invariant differential operators on $ \h^\reg $.
\end{theorem}
In general, finding nontrivial families of commuting differential operators is a difficult problem,
so it is natural to ask how to construct the operators $L_P$. An ingenious idea was proposed
by G.~Heckman \cite{Hec91}. He observed that the operators $L_P$ can be obtained by restricting
the composites of certain first order {\it differential-reflection} operators which are deformations
of the usual partial derivatives $ \partial_\xi $; the commutativity $ [L_P, L_Q]=0 $  is then an easy consequence
of the commutativity of these differential-reflection operators. The commuting
differential-reflection operators were discovered earlier by Ch.~Dunkl in his work on
multivariable orthogonal polynomials (see \cite{Dun89}). We begin by briefly recalling
the definition of Dunkl operators.

\subsubsection{Heckman's argument}
For a positive root $\alpha \in R_+$, denote by $ \hat s_\alpha $
the reflection operator acting on functions in the natural way:
$$
\hat s_\alpha :\ \c[\h^\reg] \to \c[\h^\reg]\ ,\quad  f(x) \mapsto f(s_\alpha (x)) \ .
$$
For a nonzero $ \xi \in \h$, the Dunkl operator $\,\nabla_\xi(c):\c[\h^\reg] \to \c[\h^\reg ] $
is then defined using the formula\footnote{Strictly speaking, Dunkl introduced his operators
in a slightly different form that is equivalent to \eqref{dodef} up to conjugation.}
\begin{equation}
\label{dodef}
\nabla_\xi(c) := \partial_\xi + \sum_{\alpha \in R_+} c_\alpha \, \frac{(\alpha,\xi)}{(\alpha,x)}\hat{s}_\alpha\ ,
\end{equation}
where $ \alpha \mapsto c_\alpha $ is the multiplicity function introduced in Section~\ref{opcm}.
\begin{theorem}[\cite{Dun89}]\la{Thd}
The operators $\nabla_\xi$ and $\nabla_\eta$ commute for all $\xi, \eta \in \h$.
\end{theorem}
We will sketch the proof of Theorem~\ref{Thd} (in the more general setting of complex reflection groups)
in the end of this lecture. Now, we show how the Dunkl operators can be used
to construct the commuting differential operators $ L_P $.

First, observe that the operators \eqref{dodef} transform under the action of $W$ in the same way as
the partial derivatives, i.~e. $ \hat{w} \,  \nabla_\xi = \nabla_{w(\xi)} \,\hat{w}\,$
for all $\, w \in W \,$ and $\, \xi \in \h \,$.
Hence, for any $W$-invariant polynomial $ P \in \c[\h^*]^W $, the differential-reflection operator
$ P(\nabla_\xi):\, \c[\h^\reg] \to \c[\h^\reg] $ preserves the subspace
$\, \c[\h^\reg]^W $ of $W$-invariant functions. Moreover, it is easy
to see that $ P(\nabla_\xi) $ acts on this subspace as a {\it purely differential} operator.
Indeed, using the obvious relations $\,\hat{s}_\alpha \, \partial_\xi = \partial_{s_\alpha(\xi)}
\, \hat{s}_\alpha\,$ and $ \hat{s}_\alpha \, \hat{f} = \hat{s}_\alpha(f) \, \hat{s}_\alpha \,$,
one can move all nonlocal operators (reflections) in $ P(\nabla_\xi) $ `to the right'; then
these nonlocal operators will act on invariant functions as the identity. Thus, for $ P \in \c[\h^*]^W $,
the restriction of $ P(\nabla_\xi) $ to $ \c[\h^\reg]^W $ is a $W$-invariant differential operator
which we denote by $ \mathrm{Res}^W [P(\nabla_\xi)]  \in \D(\h^\reg)^W $. If $ P(x) = |x|^2 $
is the quadratic invariant in $  \c[\h^*]^W  $, an easy calculation shows that
$$
\mathrm{Res}^W |\nabla_\xi|^2 =
\Delta_\h - \sum_{\alpha \in R_+} \frac{c_\alpha(c_{\alpha} + 1)(\alpha,\alpha)}{(\alpha,x)^2} \ ,
$$
which is exactly the generalized Calogero-Moser operator $ H_\g $ introduced in \cite{OP83}.
In general, if we set
\[
L_P := \mathrm{Res}^W [P(\nabla_\xi)] \ ,\quad \forall\,P \in \c[\h^*]^W  ,
\]
then all properties of Theorem~\ref{OPgen} are easily seen to be satisfied. In particular,
the key commutativity property $ [L_P, L_Q] = 0 $ follows from the equations
\begin{eqnarray*}
\mathrm{Res}^W[P(\nabla_\xi)] \ \mathrm{Res}^W[Q(\nabla_\xi)] &=& \mathrm{Res}^W[P(\nabla_\xi)\, Q(\nabla_\xi)] \\
&=& \mathrm{Res}^W[Q(\nabla_\xi)\, P(\nabla_\xi)] \\
&=& \mathrm{Res}^W[Q(\nabla_\xi)] \ \mathrm{Res}^W[P(\nabla_\xi)]\ ,
\end{eqnarray*}
where $ P(\nabla_\xi)\,Q(\nabla_\xi) =  Q(\nabla_\xi)\,P(\nabla_\xi) $ is a direct consequence of
Theorem~\ref{Thd}. Thus, the algebra homomorphism of Theorem~\ref{OPgen} can be defined by the rule
$$
\c[\h^*]^W \to \D(\h^\reg)^W\ ,\quad P \mapsto \mathrm{Res}^W[P(\nabla_\xi)]\ .
$$
Despite its apparent simplicity this argument looks puzzling. What kind of an algebraic structure is hidden
behind this calculation? This was one of the questions that led to the theory of rational Cherednik algebras
developed in \cite{EG02}. In the next lecture, we will introduce these algebras in greater generality:
for an arbitrary complex (i.e., not necessarily Coxeter) reflection group. The representation theory of
rational Cherednik algebras is remarkably analogous to the classical representation theory of semisimple Lie algebras,
and we will discuss this analogy in the subsequent lectures. We will also emphasize another less evident analogy
that occurs only for {\it integral} multiplicity values. It turns out that for such values, the rational Cherednik algebras are closely related to the rings of (twisted) differential operators on certain singular algebraic varieties with `good' singularities ({\it cf.} Proposition~\ref{is}).

\subsection{Finite reflection groups}
\la{frg}
The family of rational Cherednik algebras is associated to a finite reflection group. We begin by recalling the definition of such groups and basic facts from their invariant theory. The standard reference for this material is
N.~Bourbaki's book \cite{Bou68}.

Let $V$ be a finite-dimensional vector space over a field $k$. We define two distinguished classes of linear automorphisms of $ V $.
\begin{enumerate}
\item $ s \in \GL_k(V)$ is called a \emph{pseudoreflection}
if it has finite order $ > 1 $ and there is a hyperplane $ H_s \subset V$ that is pointwise fixed by $s$.
(If $s$ is diagonalizable, this is the same as saying that all but one of the eigenvalues are equal to $1$.)
\item A pseudoreflection is a \emph{reflection} if it is
diagonalizable and has order 2. (In this case, the remaining eigenvalue is equal to $-1$.)
\end{enumerate}

A \emph{finite reflection group} on $V$ is a finite subgroup $ W \subset \GL_k(V)$ generated by pseudoreflections.
Note that if $ W_1 $ is a finite reflection group on $V_1$ and $ W_2 $ is a finite reflection group on $V_2 $,
then $ W = W_1 \times W_2 $ is a finite reflection group on $ V_1 \oplus V_2 $. We say that $W$ is
indecomposable if it does not admit such a direct decomposition.

If $\,k= {\mathbb R} \,$, then all pseudoreflections are reflections. The 
classification of finite groups generated by reflections was obtained in this case by H.~S.~M.~Coxeter \cite{Cox34}. There are
4 infinite families of indecomposable Coxeter groups ($ A_n,\,B_n,\, D_n $ and $ I_2(m) $) and 6 exceptional groups
($E_6, \,E_7,\, E_8,\,F_4,\,H_3,\, H_4$). Over the complex numbers $ k = \c $, the situation is much more complicated. A complete list of indecomposable complex reflection groups was given by
G.~C.~Shephard and J.~A.~Todd in 1954: it includes 1 infinite family $ G(m,p,n) $ depending on 3 positive integer parameters (with $p$ dividing $m$), and 34 exceptional groups (see \cite{ST54}).
A.~Clark and J.~Ewing used the Shephard-Todd results to classify the groups
generated by pseudoreflections over an arbitrary field of characteristic coprime to the
group order (see \cite{CE74}).

There is a nice invariant-theoretic characterization of finite reflection groups. Namely, the invariants of
a finite group $ W \subset \GL_k(V) $ form a polynomial algebra if and only if $W$ is generated by
pseudoreflections. More precisely, we have 
\begin{theorem}
\la{sct}
Let $V$ be a finite-dimensional faithful representation of a finite group $W$ over
a field $k$. Assume that either $ {\rm char}(k) = 0$  or $ {\rm char}(k) $ is coprime to $ |W|$.
Then the following properties are equivalent:
\begin{enumerate}
\item $W$ is generated by pseudoreflections,
\item $ k[V]$ is a free module over $k[V]^W$,
\item $k[V]^W = k[f_1,\,\ldots\,,\,f_n] $ is a free polynomial algebra\footnote{The polynomials
$ f_i \in k[V] $ generating $ k[V]^W $ are not unique; however, their degrees $ d_i = \deg(f_i) $
depend only on $W$: they are called the fundamental degrees of $W$.}.
\end{enumerate}
\end{theorem}
Over the complex numbers, this theorem was originally proved by Shephard and Todd \cite{ST54}
using their classification of pseudoreflection groups. Later Chevalley \cite{Che55} gave a proof in the real case without using the classification, and Serre extended Chevalley's argument to the complex case. Below, we will give a homological proof of the (most interesting) implication $ (1)\,\Rightarrow\,(2)\,$, which is due to L.~Smith \cite{Sm85}; it works in general in coprime characteristic.

First, we recall the following well-known result from commutative algebra.
\begin{lemma}\la{NL}
Let $\, A = \bigoplus_{i \ge 0}\, A_i \,$ be a non-negatively graded commutative $k$-algebra with $ A_0 = k $,
and let $M$ be a non-negatively graded $A$-module. Then the following are equivalent:
\begin{enumerate}
\item $M$ is free,
\item $M$ is projective,
\item $M$ is flat,
\item $\tor_1^A(k,M) = 0$, where $\, k = A/A_{+} \,$.
\end{enumerate}
\end{lemma}
Note that the only nontrivial implication in the above lemma is: $\,(4)\,\Rightarrow\, (1)\,$.
Its proof is standard homological algebra (see, e.g., \cite{Se00}, Lemma~3, p. 92).

\begin{proof}[Proof of Theorem~\ref{sct}, $\,(1)\,\Rightarrow\,(2)\,$]
In view of Lemma~\ref{NL}, it suffices to prove that
\[
\tor^{k[V]^W}_1(k,\,k[V]) = 0\ .
\]
Note that the non-negative grading and $W$-action on $k[V]$ induce a non-negative grading and $W$-action on
$\,\tor_1^{k[V]^W}(k,k[V])$. There are two steps in the proof. First, we show that the nonzero elements in
$\tor_1^{k[V]^W}(k,k[V])$ of minimal degree must be $W$-invariant, and then we show that
$W$-invariant elements must be $0$.

Given a pseudoreflection $s \in W$, choose a linear form $ \alpha_s \in V^* \subset k[V]$ such that
$ \Ker (\alpha_s) = H_s $, and define the following operator
\[
\Delta_s := \frac{1}{\alpha_s}(1-s) \in \End_{k[V]^W}(k[V])\ .
\]
To see that this operator is well defined on $ k[V] $, expand a polynomial $f \in k[V]$ as a Taylor series in $\alpha_s$
and note that the constant term of $(1-s)\cdot f$ is 0. To see that $ \Delta_s $ is a $k[V]^W$-linear endomorphism,
check the identity
$$
\Delta_s(fg) = \Delta_s(f) g + s(f)\Delta_s(g)\ ,\quad \forall\,f,\,g \in k[V]\ ,
$$
and note that, for $f \in k[V]^W$, this identity becomes
$\Delta_s(fg) = f\Delta_s(g)$. It is clear that $\Delta_s$ has degree $-1$. Now, since $\tor_1^{k[V]^W}(k,-)$ is a functor on the category of $ k[V]^W$-modules, we can apply it to the endomorphism $ \Delta_s \,$:
\[
(\Delta_s)_*:\ \tor_1^{k[V]^W}(k,k[V]) \to \tor_1^{k[V]^W}(k,k[V])\ .
\]
If $\,\xi \in \tor_1^{k[V]^W}(k,k[V])\,$ has minimal degree, then $ \, (\Delta_s)_*\,\xi = 0\,$ (since
$ (\Delta_s)_* $ must decrease degree by one).
This implies that $\,s(\xi) = \xi\,$, which proves the claim that the nonzero elements of $ \tor_1^{k[V]^W}(k,k[V]) $
of minimal degree are $W$-invariant.

Now, consider the averaging (Reynolds) operator\footnote{Note that this operator is well defined since the
characteristic of $k$ is coprime to $\lvert W \rvert$.}
\[
\pi = \frac 1 {\lvert W \rvert} \sum_{g \in W} g: \, k[V] \to k[V]^W .
\]
If we write $\,\iota:\,k[V]^W \into k[V]\,$ for the natural inclusion, then
$\,\iota \circ \pi(f) = f \,$ for all $\,f \in k[V]^W $.
The identity $\pi(\pi(f)g) = \pi(f)\pi(g)$ shows that $\pi$ is a map of $k[V]^W$-modules.
Applying then the functor $ \tor_1^{k[V]^W}(k,-) $ to the composition $\, \iota \circ \pi \,$, we get
$$
\tor_1^{k[V]^W}(k, k[V]) \xrightarrow{\pi_*} \tor_1^{k[V]^W}(k, k[V]^W)
\xrightarrow{\iota_*} \tor_1^{k[V]^W}(k, k[V])
$$
which is the {\it zero} map since $\, \tor_1^{k[V]^W}(k,\, k[V]^W) = 0 \,$. On the other hand,
if $\,\xi \in \tor_1^{k[V]^W}(k, k[V])^W \,$ is $W$-invariant, we have
$ \xi = (\iota \circ \pi)_*\,\xi = (\iota_* \circ \pi_*)\,\xi = 0 $. Hence
$ \tor_1^{k[V]^W}(k, k[V])^W = 0 $ finishing the proof.
\end{proof}

\subsection{Dunkl operators}
\la{Sect4}
From now on, we assume that $ k = \c $. We fix a finite-dimensional complex vector space $V$
and a finite reflection group $W$ acting on $V$ and introduce notation for the following objects:

\begin{enumerate}
%\item $W$ is a finite reflection group on $V$,
\item $ \A$ is the set of hyperplanes fixed by the generating reflections of $W$,
\item $W_H$ is the (pointwise) stabilizer of a reflection hyperplane $H \in \A$; it is a cyclic subgroup of $W$
    of order $ n_H \ge 2 $,
\item $(\,\cdot\,,\,\cdot\,):\, V\times V \to \c $ is a positive definite Hermitian form on $V$, linear in the second factor and antilinear in the first,
\item $x^* := (x,\,\cdot\,) \in V^*$ which gives an antilinear isomorphism $ V \to V^* $,
\item for $H \in \A$, fix $v_H \in V$ with $(v_H,x) = 0$ for all $x \in H$, and $\alpha_H \in V^*$ such that $H = \Ker(\alpha_H)$,
\item $\delta := \prod_{H \in \A} \alpha_H \in \c[V]$ and $\delta^* := \prod_{H \in \A} v_H \in \c[V^*]$,
\item $ \e_{H,i} := \frac 1 {n_H} \sum_{w \in W_H}\det_V(w)^{-i}\, w \in \c W_H \subset \c W$, $\,i=0,1,\ldots, n_H-1\,$.
\end{enumerate}
Some explanations are in order. The elements $ e_{H,i} $ are idempotents which are generalizations of the primitive idempotents $ (1-s)/2 $ and $(1+s)/2$ for a real reflection $s$. The orders $ n_H = |W_H| $ depend only on the orbit of $H$ in $\A$. Under the action of $W$, $\delta^*$ transforms as the determinant character $ \det_V:\,\GL(V) \to \c $ and $\delta $ transforms as the inverse determinant. Finally, to simplify notation for sums, we introduce the convention of conflating $W$-invariant functions on $\A$ with functions on $\A/W$, and we write indices cyclicly identifying the index sets
$\{0,1,\ldots, n_H-1\}$ with $ \Z/n_H \Z $. 

The rational Cherednik algebras associated to $W$ will depend 
on complex parameters
$\,k_C := \{k_{C,i}\}_{i=0}^{n_C-1}$, where $ C $ runs over 
the set of orbits $\A/W$; these parameters play the role of multiplicities $ c_\alpha $ in the complex case.
We will assume that $k_{C,0} = 0 $ for all $ C \in \A/W$. 
We will need to extend the ring of differential operators $ \D = \D(V_\reg) $ on $ V_\reg := V\setminus \left(\cup_{H \in \A} H\right)$ by allowing group-valued coefficients. The group $W$ acts naturally on $\D$, so we simply let $\D W := \D \rtimes W $ be the crossed product\footnote{As a reminder, if $W$ is a finite group acting on an algebra $A$ by algebra automorphisms, then the crossed product $A \rtimes W$ is defined to be the vector space $ A \otimes \c W$ with multiplication $\, (a \otimes w) \cdot (b \otimes v) = a \,w(b) \otimes w v\,$.}. Then, for each
$\xi \in V$, we define the operator $T_\xi(k) \in \D W$ by
\[
T_\xi(k) := \partial_\xi - \sum_{H \in A}\frac{\alpha_H(\xi)}{\alpha_H} \sum_{i = 0}^{n_H-1} n_H k_{H,i} e_{H,i}
\]
It is an exercise in decoding the notation to check that in the case when $W$ is a Coxeter group, the 
operator $ T_\xi(k) $ agrees with the operator defined in \eqref{dodef}, up to conjugation by the function 
$\,\prod_{H \in \A} \alpha_H^{k_H} $.

The following result was originally proved by Dunkl \cite{Dun89}
in the Coxeter case and in \cite{DO03} in general. We end this lecture by outlining the proof given in \cite{DO03}. 
\begin{lemma} 
\la{ducom}
The operators $T_\xi(k)$ satisfy the following properties:
\begin{enumerate}
\item $[T_\xi(k),T_\eta(k)]=0$, i.e. the operators commute
\item $w T_\xi(k) = T_{w(\xi)}(k) w$
\end{enumerate}
\end{lemma}
Note that one consequence of this lemma is that the linear map $ \xi \mapsto T_\xi(k)$ yields an algebra embedding $ \c[V^*] \hookrightarrow \D W$.

We think of the operators $ T_\xi(k) $ as deformations of the directional derivatives $ \partial_\xi $. One strategy for proving commutativity is to find the equations $\partial_\xi \partial_\eta - \partial_\eta \partial_\xi = 0 $ ``in nature,'' and then see if the situation in which they appear can be deformed. A natural place these equations appear is in the de Rham differential: $\d^2 f = \sum_{i<j}(\partial_i \partial_j - \partial_j\partial_i) f \d x_i \wedge \d x_j$. In this case, commutativity of the partial derivatives is equivalent to the fact that the de Rham differential squares to zero.

We now proceed to deform the de Rham differential, with the eventual goal of showing that the deformed map {\it is} actually a differential. It turns out that in addition to deforming the de Rham differential, it is also convenient to deform the Euler derivation. First, for 
each $H \in \A $, define
\[
a_H(k) := \sum_{i=0}^{n_H-1} n_H k_{H,i} e_{H,i} \in \C W_H
\]
Let $K^\bullet = \C[V] \otimes \Lambda^\bullet V^*$ be the polynomial de Rham complex on $V$, and define $\Omega(k) \in \End(\C[V])\otimes K^1$ by
\[
\Omega(k) := \sum_{H \in A}a_H(k) \alpha_H^{-1} \d \alpha_\H
\]
The term $\alpha_H^{-1}\d \alpha_H$ is the logarithmic differential of $\alpha_H$, and $\Omega(k)$ is $W$-equivariant with respect to the diagonal action on $K^1$. Next we define $\d(k):\C[V] \to K^1$ via
\[
\d (k)(p) := \d p + \Omega(k)(p)
\]
(where $\d(0)$ is the standard de Rham differential). We extend it to $\d(k):K^\bullet \to K^{\bullet + 1}$ in the usual way - if $p \otimes \omega \in \C[V]\otimes \Lambda^\bullet V^*$, then $\d (k)(p\otimes \omega) := d(k)(p)\wedge \omega$.
Next, we define the standard Koszul differential $\partial:K^l \to K^{l - 1}$ using
\[
p \otimes \d x_1 \wedge \cdots \wedge \d x_l \mapsto \sum_{r=1}^l (-1)^{r+1} x_r p\otimes \d x_1 \wedge \cdots \wedge \d \hat{x}_r \wedge \cdots \wedge \d x_l
\]
Finally, we define the (deformed) Euler vector field
\[
E(k) := E(0) + \sum_{H \in A} a_H (k)
\]
where $E(0)$ is the usual Euler derivation (i.e. the infinitesimal generator of the diagonal $\C^\times$ action on $K^\bullet$). The space $K^\bullet$ has a bigrading given by $K^\bullet = \oplus_m\oplus_{l=0}^{\dim V} \C[V]_m \otimes \Lambda^l V^*$, and if we let $K_m^l$ be one of the graded pieces and let $\omega \in K^l_m$, then
\[
E(0)\omega = (l+m)\omega,\quad \d (k)\omega \in K^{l+1}_{m-1},\quad \partial \omega \in K^{l-1}_{m+1}
\]
The deformations $E(k)$ and $\d (k)$ are compatible in the following sense:
\[
E(k) = \partial \d (k) + \d (k) \partial
\]

The following important lemma is used to determine conditions on $k$ which are sufficient to show $E(k)$ has a small kernel.
\begin{lemma}[\cite{DO03}, Lemma 2.5]
Let $z(k) = \sum_{H \in A} a_H(k) \in \C W$. Then
\begin{enumerate}
\item The element $z(k)$ is central in $\C W$.
\item For an irreducible representation $\tau$ of $\C W$, let $c_\tau(k)$ denote the (unique) eigenvalue for $z(k)$ on $\tau$. Then $c_\tau(k)$ is a linear function of $k$ with non-negative integer coefficients.
\end{enumerate}
\end{lemma}

The following theorem is the main step in the proof of commutativity of the Dunkl operators.
\begin{theorem}[\cite{DO03}, Theorem~2.9]
Assume that $k$ is a parameter such that $-c_\tau(k) \not\in\N$ for all irreducible $\tau$. Then there exists a unique $W$-invariant linear isomorphism $S(k):K^\bullet \to K^\bullet$ satisfying the properties
\begin{enumerate}
\item $S(k)(K_m^l) \subset K_m^l$,
\item The restriction of $S(k)$ to $K_0^0$ is the identity,
\item $S(k)(p\otimes \omega) = (S(k)(p))\otimes \omega$,
\item $\d(k)S(k) = S(k)d(0)$.
\end{enumerate}
\end{theorem}
The \emph{intertwining operator} $S(k)$ is constructed by induction on $m$ and $l$. It is not too difficult to show that the enumerated conditions imply that $S(k)$ is unique and invertible. The main point is to show the existence of $S(k)$. The proof of existence is essentially linear algebra using the fact that the assumption on the parameter $k$ implies (via the previous lemma) that the kernel of $E(k)$ is exactly $K^0_0$.

\begin{cor}
The map $\d(k)$ is a differential on $K^\bullet$.
\end{cor}
\begin{proof}
If we assume $-c_\tau(k)\not\in \N$ for all irreducible $\tau$, then this corollary follows immediately from the existence of the intertwining operator (and the fact that $\d(0)$ is a differential). However, the condition $\d^2(k)=0$ is a closed condition (either in Zariski or classical topology), so it must hold on a closed subset of the space of all parameters $k$. Since the theorem holds on an open dense set in this space, it implies the corollary for all parameter values.
\end{proof}
As mentioned above, the commutativity of Dunkl operators
(Lemma~\ref{ducom}) follows directly from this corollary. Indeed, let $e_i \in V$ and $x_i \in V^*$ be dual bases, and write $T_i = T_{e_i}(k)$. A straightforward computation shows that for all $f \in \c[V] $,
\[
\d(k)^2 f=\sum_{i<j}(T_iT_j-T_jT_i)f\otimes \d x_i\wedge \d x_j\ .
\]
A different proof of commutativity can be found in \cite{EG02}, Section~4.
\begin{remark}
At first glance, the definition of the Dunkl operators may seem to be quite general. Indeed, at least in the real case, the formulas
defining $ T_\xi(k) $ make sense for an \emph{arbitrary} finite hyperplane arrangement $\A$ (with prescribed multiplicities). However, the operators $ T_\xi $ thus defined will commute if and only if the hyperplane arrangement $ \A $ comes from a finite reflection group ({\it cf.} \cite{Ve94}).
\end{remark}

\section{Lecture 2}
In this lecture, we introduce the rational Cherednik algebras
and discuss their basic properties. 
We also define category $\O$, which is a category of modules over the Cherednik algebra subject to certain finiteness restrictions. This is a close analogue of the eponymous category of $\g$-modules defined by J.~Bernstein, I.~Gelfand, and S.~Gelfand for a semi-simple Lie algebra $\g$ (see \cite{Hum08}).

\subsection{Rational Cherednik algebras} In this section, we will use the notation introduced in Section~\ref{Sect4}.
We begin with the main definition.
\begin{defi}[{\it cf.} \cite{DO03}]
The \emph{rational Cherednik algebra} $H_k(W)$ is the subalgebra of $\D W$ generated by $\c[V]$, $\c[V^*]$ and $\c W$. (The subalgebras $\c[V]$ and $\c W$ are embedded in $ \D W $ in the natural way and are independent of $k$. On the other hand, the embedding of $\c[V^*]$ in $ \D W $ is defined via the Dunkl operators $T_\xi(k)$ which certainly depend on $k$.)
\end{defi}

It is also possible to give an `abstract' definition of Cherednik algebras in terms of generators and relations. 
From this point of view, the previous definition is called the {\it Dunkl representation}.
The key point is that the Dunkl representation is faithful. 
The algebra $ H_k(W) $ is generated by the elements of $V$, $V^*$ and $ W$ subject to the following relations
\begin{eqnarray*}
&&[x,\,x'] =0\ ,\quad [\xi,\,\xi']=0 \ ,\quad
w\,x\,w^{-1} =w(x)\ ,\quad w\,\xi\,w^{-1}=w(\xi)\ ,\nonumber \\*[1ex]
&& [\xi,\,x] = \langle \xi,\,x \rangle + \sum_{H\in \A}\frac{\langle \alpha_H, \xi\rangle\,
\langle x, v_H\rangle}{\langle \alpha_H, v_H\rangle}\sum_{i=0}^{n_H-1}n_H(k_{H,i}-k_{H,i+1})\,e_{H,i}\ .
\end{eqnarray*}
where $x,x' \in V^*$ and $\xi,\xi' \in V$ and $w \in W$.
\begin{example}
In the case $ W = \Z_2 $, the above relations are actually
very simple. Specifically, $H_k(\Z_2)$ is generated by $x$, $\xi$ 
and $s$ satisfying
\[
s^2=1\ , \quad s x=- x s\ , \quad s \xi=-\xi s\ , \quad 
[\xi,x]= 1- 2ks\ .
\]
The Dunkl operator corresponding to $\xi$ is given by 
$\frac d {dx} - \frac k x (1-s)$ which acts naturally on $V_\reg = \c[x,x^{-1}]$, 
preserving $ \c[x] \subset \c[x, x^{-1}]$.
\end{example}

\subsection{Basic properties of $H_k(W)$}
First, note that if $k=0$, then $H_0 = \D(V)\rtimes W \subseteq \D(V_\reg)\rtimes W$, so we can view the Cherednik algebra as a deformation of the crossed product of a Weyl algebra $\D(V)$ with the group $W$.
The next theorem collects several key properties of $H_k(W)$.
\begin{theorem}[see \cite{EG02}]
\label{keyprop} 
Let $ H_k = H_k(W) $ be the family of Cherednik algebras associated to a complex reflection group $W$.
\begin{enumerate}
\item\la{universality} {\rm Universality:} $\{H_k\}$ is the universal deformation of $H_0$.
\item {\rm PBW property:} the linear map $\c[V]\otimes \c W\otimes \c[V^*] \to H_k$ induced by multiplication in $H_k$ is a $\c[V]$-module isomorphism.
\item Let $H_\reg = H_k[\delta^{-1}]$ denote the localization of $ H_k $ at the Ore subset 
$\{\delta^k\}_{k \in \mathbb N}$. Then the induced map $ H_\reg \to \D W $ is an isomorphism of algebras.
\end{enumerate}
\end{theorem}
\remark Each statement deserves a comment:
\begin{enumerate}
\item In general, universal deformations can rarely be realized algebraically, so the family of 
Cherednik algebras is somewhat exceptional.

\item The name `PBW property' comes from the Poincare-Birkhoff-Witt Theorem in Lie theory, which has a very similar statement. It is a fundamental property for many reasons. In particular, it is not obvious {\it a priori} that the generators and relations listed above give a nonzero algebra, so in some sense the important part of this statement is the ``lower bound'' on the size of $H_k$. (The ``upper bound'' is easy to see from the relations.)
\item This justifies the notation $H_\reg := H_k[\delta^{-1}]$, since it shows that the localization is independent of the parameters $k$.
\end{enumerate}

There are two filtrations on $\D W$ which are commonly used. The first is the \emph{standard filtration}, where $\deg x = 1 = \deg\xi$ and $\deg w = 0$. The second is the \emph{differential filtration}, where $\deg x = 0 = \deg w$, and $\deg \xi = 1$. Through the Dunkl embedding $H_k \into \D W$, these filtrations induce filtrations on $H_k$ (with the same names), and for both filtrations we have $\grd (H_k) \cong \c[V \oplus V^*] \rtimes W$.

\subsection{The spherical subalgebra}
Each $H_k$ contains a distinguished (nonunital)
subalgebra $U_k = U_k(W)$ called the \emph{spherical subalgebra}. We set $\e := \frac 1 {\lvert W \rvert} \sum_{w \in W} w \in \D W$ and define
\[
U_k := \e H_k \e
\]
(The additive and multiplicative structure of $\e H_k \e$ are induced from $H_k$, but the unit of $\e H_k \e$ is $\e$, not $1\in H_k$.) 

The spherical subalgebra $U_k$ is closely related to $H_k$, however the exact relationship depends crucially on values 
of the parameter $k$. For each $k$, there is an algebra isomorphism $ U_k \cong \End_{H_k}(\e H_k)$, and although $ \e H_k $ is a f.~g. projective module over $H_k$, this does not imply that $U_k$ and $H_k$ are Morita equivalent\footnote{We recall that two rings are \emph{Morita equivalent} if their categories of left (or right) modules are equivalent. We refer to \cite{MR01}, Section~3.5, for basic Morita theory.}. The problem is that for certain special values of $k$, the module $ \e H_k $ is not a generator in the category of right $H_k$-modules. Such special values are called {\it singular}, 
and it is an interesting (and still open) question to precisely determine these values of $k$ for a given 
$W$. For \emph{generic} $k$, one can prove that $ H_k $ is a simple algebra (see Theorem~\ref{simp} below), so
in that case, the algebras $ H_k $ and $ U_k $ are Morita equivalent.

For $k=0$, we know that $H_0 = \D(V)\rtimes W$. Since $\D(V)$ is isomorphic to a Weyl algebra, it is simple, which implies that $\D(V)\rtimes W$ is simple. This implies, in turn, that $ H_0 $ is Morita equivalent to $ U_0$.
We also remark that $U_0 = \e (\D(V) \rtimes W)\e \cong \D(V)^W $ via the identification $\e d \e \leftrightarrow d $. 
This allows a theorem analogous to Theorem \ref{keyprop}(\ref{universality}):
\begin{theorem}
The family $\{U_k\}$ is a universal deformation of $\D(V)^W$. Also, we have $\grd (U_k) \cong \c[V \oplus V^*]^W$.
\end{theorem}

In general, the algebra map $ H_k \into \D W$ restricts to a map $U_k = \e H_k \e \to \e (\D W)\e$ which is an injective homomorphism of unital algebras. Heckman's restriction operation now becomes
\begin{equation}\label{HC}
\Res_k^W:\ U_k = \e H_k \e \hookrightarrow \e (\D W )\e \stackrel{\sim}{\to} \D(V_\reg)^W
\end{equation}
where the second map is an isomorphism given by $\e D \e \mapsto D$. We call $\Res_k^W:U_k \hookrightarrow \D(V_\reg)^W$ the \emph{spherical Dunkl representation} of $U_k$. Note that $\Res_k^W$ is a deformation of the canonical embedding $\D(V)^W \hookrightarrow \D(V_\reg)^W$ in the same way as the Dunkl embedding $H_k \hookrightarrow \D W$ is a deformation of the canonical map $ \D(V) \rtimes W \into \D W$.

\subsection{Category $\O$}
In this section we discuss a subcategory of $H_k$-modules that shares many properties with the Berstein-Gelfand-Gelfand category $\O$ in Lie theory. (The BGG category $\O$ is a subcategory of representations of the universal enveloping algebra $\U(\g)$ which are subject to certain finiteness conditions. A good exposition of this theory can be found
in \cite{Hum08}.) Some analogies between $\U(\g)$ and $H_k(W) $ are listed in Table \ref{analogies}.

\begin{table}
\begin{tabular}{|p{5cm}|p{5cm}|}
\hline
$\U(\g) = \U(\n_-)\otimes \U(\h)\otimes \U(\n_+)$ & $H_k(W) = \c[V] \otimes \c W\otimes \c[V^*]$\\
\hline
Weights for $\U(\g)$ are irreducible $\U(\h)$-modules, i.e. $\mu \in \h^*$ & `Weights' for $H_k$ are irreducible $W$-modules\\
\hline
central character of $M$ & the parameter $k$ \\
\hline
the BGG category $\O$&category $\O'_k$\\
\hline
blocks $\O_\chi$ of $\O$ & blocks $\O_k(\bar\lambda)$\\
\hline
\end{tabular}
\caption{Analogies between $\U(\g)$ and $H_k$}\label{analogies}
\end{table}
%\marginpar{**} %is the last row of the table correct? Also, I'm confused about the \lambda's, they should be checked.

\begin{defi} We introduce the following subcategories of the category $ \mod(H_k)$ of finitely generated $H_k$-modules:
\begin{eqnarray*}
\O'_k &:=& \{M \in \mathrm{mod}(H_k) \mid \dim_\c (\c[V^*]\cdot m) < \infty\}\\
\O_k(\bar\lambda) &:=& \{M \in \O'_k \mid (P-\bar\lambda(P))^N \cdot m = 0, \; N  \gg 0\}\\
\O_k &:=& \{M \in \O'_k \mid \xi^N\cdot m = 0,\; N \gg 0\} = \O_k(0)
\end{eqnarray*}
\noindent
In these definitions, $ \bar \lambda \in V^*/ W$, (i.e. $\bar \lambda:\c[V]^W \to \c$), and the conditions 
are to hold for arbitrary elements $ m $, $P$, and $\xi$ (where $m \in M$ and $P \in \c[V^*]^W$ and $\xi \in V$):
\end{defi}

From now on, we mainly discuss $\O_k$, which is called the \emph{principal block} of category $\O'_k$. The following lemma is standard and closely mirrors the Lie situation.
\begin{lemma} The objects of $\O'_k$ and $\O_k$ have the following properties:
\begin{enumerate}
\item Each $M \in \O'_k$ (resp. $M \in \O_k$) is finitely generated over $\c[V] \subset H_k$.
\item $\O'_k$ (resp. $\O_k$) is a stable Serre subcategory of $\mathrm{mod}(H_k)$ (i.e. is it an abelian subcategory closed under taking subobjects, quotients, and extensions).
\item $\O'_k$ (resp. $\O_k$) is Artinian.
\end{enumerate}
\end{lemma}

Just as in the Lie case, the most important objects in $\O_k$ are obtained by inducing modules from 
subalgebras of $H_k$. In particular, fix an irreducible representation $\tau$ of $W$, and give it a $\c[V^*]$ module structure using $P\cdot x = P(0)x$ for $P \in \c[V^*]$ and $x \in \tau$. Since this action is $W$-invariant, it gives $\tau$ the structure of a $\c[V^*]\rtimes W$-module. We then define the \emph{standard module} associated to $ \tau $ by
\[
M(\tau) := \mathrm{Ind}_{\c[V^*]\rtimes W}^{H_k}(\tau) = H_k \bigotimes_{\c[V^*]\rtimes W}\tau
\]
These standard modules are analogues of Verma modules in Lie theory.
(Of course, we can also induce other modules of $\c[V^*]\rtimes W$, but since we are working with the principal block $\O_k$, this definition is sufficient for our purposes.) The PBW property for $H_k$ shows that $M(\tau) \cong \c[V] \otimes \tau$ as a $\c[V]$-module. Also, the relations in $H_k$ make it clear that $M(\tau) \in \O_k$.
\begin{theorem}
As in the Lie case, the following properties hold.
\begin{enumerate}
\item The set $\{M(\tau)\}_{\tau \in \mathrm{Irr}(W)}$ is a complete list of pairwise non-isomorphic indecomposable objects in $\O_k$.
\item Each $M(\tau)$ has a \emph{unique} simple quotient $L(\tau)$, and the $L(\tau)$ are a complete list of simple objects in $\O_k$
\item The Jordan-H\"older property: Each $M \in \O_k $ has a finite filtration whose associated graded module is isomorphic to a sum of the $L(\tau)$ and is independent of the filtration.
\item $\O_k$ is a highest weight category {\rm (}in the sense of \cite{CPS88}\rm{)}.
\end{enumerate}
\end{theorem}
For the proof of this theorem, we refer to \cite{DO03}, Proposition~2.27, and \cite{GGOR03}, Proposition~2.11
and Corollary~2.16. 

\section{Lecture 3}
The category $\O$ defined at the end of the last lecture is related to the category of finite-dimensional 
representations of the \emph{Iwahori-Hecke algebra} associated to $W$. The relation is given by a certain
additive functor (called the {\it KZ functor}) which plays a fundamental role in representation theory of $ H_k$.
In this lecture, we will introduce this functor and discuss some of its applications. In particular,
we define certain operations (\emph{KZ twists}) on the set of isomorphism classes of irreducible $W$-modules 
depending on the (integral) parameter $k$ and discuss interesting relations between these operations
(conjectured by E.~Opdam \cite{Opd95, Opd00} and proved in \cite{BC11}).

If $A$ is an algebra, we write $\Mod(A)$ (respectively, $\mod(A)$) for the category of left (respectively, finitely generated left) $A$-modules.

\subsection{The Knizhnik-Zamolodchikov (KZ) functor}
The KZ functor is defined as a composition of several functors, and we describe each one in turn.
The key property that allows this construction is the well-known fact that a $\D_X $-module on a smooth 
algebraic variety $X$ which is coherent as an $\O_X$-module is the same thing as a 
vector bundle with a flat connection on $X$. 

The Dunkl embedding provides a natural functor $\Mod(H_k) \to \Mod(\D W)$ given by $M \mapsto \D W \otimes_{H_k} M$. We denote the output of this functor by $M_\reg$. Next, we note that $\Mod(\D W)$ is naturally equivalent to the category of $W$-equivariant $D$-modules on $V_\reg$. Since $W$ acts freely on $V_\reg$, this gives an equivalence $\Mod(\D W) \cong \Mod(\D(V_\reg/W))$. The category $\Mod(\D(V_\reg/W))$ contains a full subcategory of $\O$-coherent $\D$-modules (which are automatically $\O$-locally-free). There is an interpretation functor $\Mod_\O(\D(V_\reg/W)) \cong \mathrm{Vect}^f(V_\reg/W)$ which interprets an $\O$-locally-free $\D$-module as a vector bundle with a flat connection. Finally, the Riemann-Hilbert correspondence gives an equivalence of categories $\mathrm{Vect}^f(V_\reg/W) \cong \mod(B_W)$, where $B_W := \pi_1(V_\reg/W,\, *)$ is the Artin braid group.
\begin{defi}
The $ \KZ$ \emph{functor} is defined by the composition of functors
\[
\KZ_k:\O_k \to \mod(\D W) \cong \mod(\D(V_\reg/W)) \cong \mathrm{Vect}^f (V_\reg/W) \cong \mod(\c B_W)
\]
\end{defi}

For the standard modules $M(\tau)$ this composition of functors can be made quite explicit. Since $M(\tau) \cong \c[V]\otimes \tau$ as a $\c[V]$-module, it is free of rank $\dim_\c(\tau)$, and its localization $M(\tau)_\reg$ is isomorphic to $\c[V_\reg]\otimes \tau$ as a $\c[V_\reg]$-module. This identification allows one to interpret $M(\tau)_\reg$ as (sections of) a trivial vector bundle of rank $\dim_\c \tau$. The $\D$-module structure on $M(\tau)_\reg$ is given by
\[
\partial_\xi(f \otimes v) = (\partial_\xi f)\otimes v + f\otimes(\partial_\xi v)
\]
By definition, in $M$ we have $\xi\cdot v = 0$, and since $\xi \mapsto T_\xi$ under the Dunkl embedding, we know $T_\xi(v) = 0$. Rewriting this, we see
\begin{equation}\label{kzconnection2}
\partial_\xi v - \sum_{H \in A}\frac{\alpha_H(\xi)}{\alpha_H}\sum_{i=0}^{n_H-1} n_H k_{H,i} \e_{H,i}(v) = 0
\end{equation}
Combining these formulas, we obtain
\begin{equation}\label{kzconnection}
\partial_\xi(f\otimes v) = \partial_\xi(f)\otimes v + \sum_{H \in A}\frac{\alpha_H(\xi)}{\alpha_H}\sum_{i=0}^{n_H-1} n_H k_{H,i} f \otimes \e_{H,i}(v)
\end{equation}
The right hand side is an explicit formula for the \emph{KZ connection}, which is the regular flat connection on $M_\reg = \c[V]\otimes \tau$ given by the KZ functor. (Note the change in sign between the formula for the connection and for the Dunkl operator.) Horizontal sections $y: V_\reg \to \tau$ satisfy the \emph{KZ equations}
\begin{equation}\label{kzequation}
\partial_\xi (y) + \sum_{H \in A} \frac{\alpha_H(\xi)}{\alpha_H}\sum_{i=0}^{n_H-1} n_H k_{H,i} \e_{H,i}(y) = 0
\end{equation}

\remark The formulas (\ref{kzconnection2}) and (\ref{kzequation}) look very similar (other than the sign), but there is an important distinction. In the KZ connection (\ref{kzconnection2}) group elements act on the arguments of the functions involved, while in the KZ equation (\ref{kzequation}) they act on their values.

\subsection{The Hecke algebra of $W$}
It is natural to ask what the image of the $\KZ$ functor is. The
answer is that for generic $k$, its image is the subcategory of $\c B_W$-modules that factor through a natural quotient of $B_W$ called the {\it Hecke algebra} of $W$. We first record two facts about complex reflection 
groups the proofs of which can be found in \cite{BMR98}:
\begin{enumerate}
\item For all $H \in \A$, there is a unique $s_H \in W_H$ such that $\det(s_H) = e^{2\pi i / n_H}$.
\item The braid group $B_W$ is generated by the elements $\sigma_H$ which are monodromy operators (around the $H \in \A$) corresponding to the $s_H$.

\end{enumerate}
Following \cite{BMR98}, we define the Hecke algebra of $W$ by
\[
\H_k(W) := \c B_W \bigg/ \left(\prod_{j=0}^{n_H-1} (\sigma_H - (\det s_H)^{-j}e^{2\pi i k_{H,j}})=0\right)_{H \in \A}
\]
If $k_{H,i} \in \Z$, then the relations simplify to $\sigma_H^{n_H} = 1$, which shows that there is a canonical isomorphism $\H_k(W) \cong \c W$. In particular, $\dim_\c \H_k(W) = \lvert W \rvert$. Furthermore, if $k$ is generic, then $\H_k(W)$ is semi-simple, and by rigidity of semisimple algebras
, $\H_k(W)$ is generically isomorphic to $\c W$. However, it seems to be the case that the equality $\dim_\c \H_k(W) = \lvert W \rvert$ is still conjectural for a few exceptional $W$. (The equality is known for all Coxeter groups and for all but finitely many $W$.)

\begin{theorem}[see \cite{BEG03}, \cite{GGOR03}, \cite{DO03}]
For each $k$, the functor $ \KZ_k:\O_k \to \mod(B_W)$ is an exact functor with image contained in $\mod (\H_k) \hookrightarrow \mod(B_W)$ (where the inclusion is induced by restriction of scalars).
\end{theorem}
It is natural to ask whether this can be an equivalance. One obvious obstruction to a positive answer is the fact that the localization $H_k \to \D W$ can kill some modules. It turns out that this is the only obstruction. More precisely, if we let $\O_k^\mathrm{tor} := \{M \in \O_k \mid M_\reg = 0\}$, then we have the following theorem.
\begin{theorem}[see \cite{GGOR03}]
Assume that $\dim \H_k = \lvert W \rvert$. Then
\begin{enumerate}
\item The $\KZ_k$ functor induces an equivalence $\KZ_k:\O_k/\O_k^\mathrm{tor} \to \mod(\H_k)$.
\item There is a `big projective' $P \in \mathrm{Ob}(\O_k)$ and $Q \in \mathrm{Ob}(\mod(\H_k))$ such that $\H_k \cong \End_{\O_k}(P)$ and $\O_k \cong \mod(\End_{\H_k}(Q))$.
\end{enumerate}
\end{theorem}

One can interpret this theorem as showing that the structure of the category $\O_k$ is controlled by 
$\H_k$ for all $k$. The following result shows that the {\it algebra} structure of $ H_k $ 
also depends crucially on $ \H_k $.
\begin{theorem}
\la{simp}
Assume that $\dim_\c \H_k = \lvert W\rvert$. Then the following are equivalent:
\begin{enumerate}
\item $\H_k$ is a semisimple algebra,
\item $\O_k$ is a semisimple category,
\item $H_k$ is a simple algebra.
\end{enumerate}
If one of these conditions hold, then $\O_k^\mathrm{tor} = 0$ and $\O_k \cong \mod(\H_k)$.
Furthermore, in this case the $M(\tau)$ are simple, i.e. $M(\tau) = L(\tau)$ for all $\tau \in \mathrm{Irr}(W)$.
\end{theorem}
For a detailed proof of Theorem~\ref{simp}, we refer to \cite{BC11}, Theorem~6.6, which combines
the earlier results of \cite{BEG03}, \cite{GGOR03}, \cite{DO03} and is based on R.~Vale's 
Ph.D. thesis (2006).
\begin{cor}
If the conditions of Theorem~\ref{simp} hold, then $H_k$ is Morita equivalent to its 
spherical subalgebra $ U_k$.
\end{cor}
\begin{proof}
Since $H_k$ is simple, the two-sided ideal $H_k \e H_k$ must be $H_k$, which implies that $\e H_k$ is a generator in $\mod(H_k)$. Since $ U_k \cong \End_{H_k}(\e H_k)$ and $\e H_k$ is projective, $ U_k$ and $H_k$ are Morita equivalent.
\end{proof}
When the corollary applies the mutual equivalences can be written explicitly. The functor $\mod(H_k) \to \mod(U_k)$ is given by $M \mapsto \e M := \e H_k \otimes_{H_k} M$, and the functor $\mod(U_k) \to \mod(H_k)$ is given by $N \mapsto H_k\e \otimes_{U_k} N$.

\subsection{Shift functors and KZ twists}
The goal of this section is to relate the categories $\O_k$ for different values of $k$. There are several constructions of `shift functors,' i.e. functors between different $\O_k$, the first of which was introduced in \cite{BEG03}. In this lecture we will focus on the functor introduced in \cite{BC11}. The main idea is that the Dunkl embeddings all have the same target, and we can `push forward' a module along one embedding and `pull back' along another. This is analogous to the so-called Enright completion in Lie theory (see \cite{Jos82}).

The first step is to enlarge $\O_k$ by defining $\O_k^{ln} \subset \Mod(H_k)$ to be all $H_k$-modules on which the $\xi \in V$ act locally nilpotently. Then $\iota_k:\O_k \hookrightarrow \O_k^{ln}$ is the natural inclusion whose image is the finitely generated modules. The inclusion $\O^{ln} \hookrightarrow \Mod(H_k)$ has a \emph{right} adjoint $\r_k:\Mod(H_k) \to \O_k^{ln}$ which outputs the largest submodule in $\O_k^{ln}$. In other words,
\[
\r_k(M) := \{m \in M \mid \xi^dm = 0,\,\forall \xi \in V,\, d \gg 0\}
\]
Recall $\D W = \D(V_\reg) \rtimes W$, and that the Dunkl embedding gives us an identification $H_k[\delta^{-1}] \cong \D W$. Write $\theta_k:H_k \to \D W$ for the localization map (which is just the Dunkl embedding). Also, write $\theta_k^*: \Mod(H_k) \to \Mod(\D W)$ for extension of scalars and $(\theta_k)_*:\Mod(\D W) \to \Mod(H_k)$ for the restriction of scalars.
\begin{defi}
For $k,k'$ two parameter values, define $\T_{k \to k'}:\Mod(H_k)\to \Mod(H_{k'})$
\[
\T_{k \to k'} = \r_{k'}(\theta_{k'})_*\theta_k^*
\]
\end{defi}
\begin{prop}
$\T_{k \to k'}$ restricts to a functor $\O_k \to \O_{k'}$.
\end{prop}
\begin{proof}
Given $ M \in \O_k $, let $\, N := (\theta_{k'})_* (\theta_k)^* M
\in \Mod(H_{k'}) $. To prove the claim we need only to show that $
\r_{k'}(N) $ is a finitely generated module over $ H_{k'}$. Assuming the contrary,
we may construct an infinite {\it strictly} increasing chain
of submodules $\,N_0 \subset N_1 \subset N_2 \subset\,
\ldots \subset \r_{k'}(N)\subset \Mreg$, with $N_i\in \O_{k'}$.
Localizing this chain, we get an infinite chain of
$\Hreg$-submodules of $\Mreg$. Since $ \Mreg $ is finite over $\c[\vreg]$ and $\c[\vreg]$ is
Noetherian, this localized chain stabilizes at some $i$. Thus, omitting
finitely many terms, we may assume
that $(N_i)_{\mathrm{reg}}=(N_0)_{\mathrm{reg}}$ for all $i$. In
that case all the inclusions $\,N_i \subset N_{i+1} \,$ are
essential extensions, and since each $\, N_i \in \O_{k'} \,$, the
above chain of submodules can be embedded into an injective hull of
$ N_0 $ in $ \O_{k'} $ and hence stabilizes for $ i \gg 0 $. (The
injective hulls in $ \O_{k'} $ exist and have finite length, since $
\O_{k'} $ is a highest weight category, see \cite{GGOR03},
Theorem 2.19.) This contradicts the assumption that the inclusions
are strict. Thus, we conclude that $ \r_{k'}(N) $ is finitely generated.
\end{proof}
As one may expect, the properties of the shift functor depend on the parameters $k$ and $k'$.
We call a parameter $k$ \emph{regular} if $\H_k$ is semisimple, and we write $\Reg(W)$ for the set of regular parameters. It is proved in \cite{BC11}, Lemma 6.9, that $\Reg(W)$ is a connected set.
We list some basic properties of the shift functor:
\begin{lemma}\label{twistsadd}
Let $k,k',k''$ be arbitrary complex multiplicites, and let $M \in \O_k$.
\begin{enumerate}
\item If $k \in \Reg(W)$, then $\T_{k\to k}(M) = M$.
\item If $k,k' \in \Reg(W)$ and $M$ is simple, then $\T_{k\to k'}(M)$ is either $0$ or simple.
\item If $k,k',k'' \in \Reg(W)$, then $(\T_{k'\to k''})\circ (\T_{k\to k'}) \cong \T_{k\to k''}$.
\end{enumerate}
\end{lemma}

\begin{cor}
If $k,'k \in \Reg(W)$, then the following are equivalent:
\begin{enumerate}
\item $\T_{k\to k'}[M_k(\tau)] \cong M_{k'}(\tau')$;
\item $M_k(\tau)_\reg \cong M_{k'}(\tau')$ as $H_\reg$-modules.
\end{enumerate}

\end{cor}

\subsection{KZ twists}
In this section we assume that the parameter $k$ is integral. Summarizing what we have said before, in this case we know that there is a canonical isomorphism $\H_k \cong \c W$, that the functor $\KZ_k:\O_k \to \mod(W)$ is an equivalence, and that the standard modules $M_k(\tau)$ are all simple.
\begin{defi}
The {\it KZ twist} $\,\kz_k:\mathrm{Irr}(W) \to \mathrm{Irr}(W)\,$ is the map induced by the KZ functor, 
i.e. $\,[\tau] \mapsto \KZ_k[M_k(\tau)]$.
\end{defi}
This map was defined by Opdam in \cite{Opd95} without the use of the Cherednik algebras (which hadn't
been defined at the time). He showed that the maps satisfied the additivity property $ \kz_k \circ \kz_{k'} = \kz_{k+k'} $ and conjectured that each $\kz_k$ was a permutation. This conjecture was proved as a corollary of the following theorem.

\begin{theorem}[see \cite{BC11}, Theorem 7.11] 
Let $k$ and $k'$ be complex multiplicities such that $k'_{H,i}-k_{H,i} \in \Z$ for all $H$ and $i$. Then
\begin{enumerate}
\item $\T_{k\to k'}(M) \not= 0$ for each standard module $M = M(\tau) \in \O_k$;
\item if $k,k' \in \Reg(W) $, then $\T_{k\to k'}[M(\tau)] = M_{k'}(\tau')$ with $\tau' = \kz_{k-k'}(\tau)$.
\end{enumerate}
\end{theorem}
The theorem is first proved for integral $k,k'$ and is then extended to all parameters using a deformation argument (and the fact that the integers are Zariski-dense in $\c$). The following is essentially a direct corollary of the second statement of the theorem and Lemma \ref{twistsadd}:
\begin{cor}[\cite{BC11}]
The map $k \mapsto \kz_k$ is a homomorphism from the additive group of integral multiplicities to the permutation group of $\{\mathrm{Irr}(W)\}$.
\end{cor}
This result was conjectured by E.~Opdam (see \cite{Opd95} and \cite{Opd00}).

\section{Lecture 4}
The notion of a quasi-invariant polynomial for a finite Coxeter group was introduced by O.~Chalykh and 
A.~Veselov in \cite{CV90}. Although quasi-invariants are a natural
generalization of invariants, they first appeared in a slightly disguised form (as symbols of
commuting differential operators). More recently, the algebras of quasi-invariants and associated varieties have been studied in \cite{FV02, EG02b, BEG03} by means of representation theory and have found applications in other areas. In this lecture, we define quasi-invariants for an arbitrary complex reflection group
and give new applications. This material is borrowed from \cite{BC11}.
\subsection{Quasi-invariants for complex reflection groups}
We will introduce a family of submodules
of $\c[V]$ depending on the parameter $k$ that interpolate between $\c[V]^W$ and $\c[V]$. These submodules are defined for {\it integral} values of $k$ and can be interpreted as torsion-free coherent sheaves on certain (singular) algebraic varieties. The ring of invariant differential operators on such a variety turns out to be isomorphic to a spherical subalgebra $U_k$, and the modules of quasi-invariants become (via this isomorphism) objects of
(the spherical analogue of) category $ \O_k $. This explains the comment at the end of Section~\ref{heckmansargumentsection}.

We first recall the definition of quasi-invariants for Coxeter groups from \cite{CV90}. Let $W$ be a Coxeter group with $H$ a reflection hyperplane, and recall that $s_H$ is the unique element of $W_H$ (the pointwise stabilizer of $H$ in $W$) with determinant $-1$. Then $f \in \c[V]^W$ if and only if $s_H(f) = f$ for all $H \in \A$. Let $k:\A/W \to \Z^{\geq 0}$.
\begin{defi}
The \emph{quasi-invariants} (for $W$ a Coxeter group) of parameter $k$ are defined as $Q_k(W) := \{f \in \c[V] \mid s_H(f) \equiv f\, \mathrm{mod}\langle\alpha_H\rangle^{2k_H}\}$.
\end{defi}
\noindent
Note that for extreme parameter values, $ Q_0(W) = \c[V]$ and $Q_{\infty}[W] = \c[V]^W$.

\vspace{2ex}

We now return to the generality of complex reflection groups. Recall the idempotents $\e_{H,i} = \frac 1 {n_H} \sum_{w \in W_H} (\det w)^{-i}w$ for $i=0,1,\ldots,n_H -1$.
\begin{defi}\label{qc} For an arbitrary complex reflection group $W$, 
the \emph{quasi-invariants} $Q_k \subset \c[V]$ are defined by
\[
Q_k(W) := \{f \in \c[V] \mid \e_{H,-i}(f) \equiv 0\, \mathrm{mod}\langle \alpha_H\rangle^{n_H k_{H,i}},\,\,\forall \;0\leq i\leq n_H-1,\,H\in \A\}
\]
\end{defi}

Note that in the Coxeter case, $ s_H(f) \equiv f $ is equivalent to 
$ (1 + \det(s_H)s_H)(f) \equiv 0 $, so the two definitions agree in this case. 
Also, the condition holds automatically for $i=0$, as we assumed that $k_{H,0} = 0$ for all $H \in \A$.

\begin{example}\label{1dim}
If $W = \Z/n\Z$ and $V = \c$, then $s(x) = e^{2i\pi/n}x$, the parameter $k$ is $k = \{k_0=0,k_1,\ldots,k_{n-1}\}$, and
\begin{equation}\label{exq}
Q_k(\Z/n\Z) = \bigoplus_{i=0}^{n-1}x^{nk_i+1}\c[x^n]
\end{equation}
\end{example}
In the Coxeter case, $Q_k(W)$ is a subring of $\c[V]$, but as the above example shows, this is not always true in the general case. However, there is a natural subring of $\c[V]$ associated to $Q_k(W)$, namely:
\begin{equation}\label{acc}
A_k(W) := \{P \in \c[V] \mid pQ_k(W) \subset Q_k(W)\}
\end{equation}
We write $X_k = \mathrm{Spec}[A_k(W)]$.

\begin{lemma}\label{alg}
We list some properties of $A_k$ and $Q_k$:
\begin{enumerate}
\item $A_k(W) = Q_{k'}(W)$ for some parameter $k'$. In particular, both $A_k$ and $Q_k$ contain $\c[V]^W$ and are stable under the action of $W$.
\item $A_k$ is a finitely generated graded subalgebra of $\c[V]$, and $Q_k$ is a finitely generated graded module over $A_k$ of rank one.
\item The field of fractions of $A_k$ is $\c(V)$, and the integral closure of $A_k$ in $\c(V)$ is $\c[V]$.
\item $X_k$ is an irreducible affine variety, and the normalization of $X_k$ is $\c^n$.
\item The normalization map $\pi_k: \c^n \to X_k $ is bijective\footnote{The normalization map is bijective as a map of \emph{sets}, but it is not an isomorphism of schemes.}.
\item The preimage of the singular locus of $ X_k $ under $ \pi_k $ is the divisor $ (\A,\,k) $.
\end{enumerate}
\end{lemma}

\subsubsection{$\c W$-valued quasi-invariants}
Our goal now is to show that $Q_k \subset \c[V]$ is preserved by the action of the spherical subalgebra $U_k$. However, this would be impossible to show by direct calculation, since for some complex reflection groups the minimial order of a non-constant $W$-invariant polynomial is 60. We therefore give a definition of quasi-invariants at the level of the Cherednik algebra itself, and then show that symmetrizing the $\c W$-valued quasi-invariants produces the quasi-invariants defined in the previous section.

The algebra $ \D W $ can be viewed as a ring of $W$-equivariant
differential operators on $ \vreg $, and as such it acts naturally
on the space of $ \c W$-valued functions. More precisely, using
the canonical inclusion $\, \c[\vreg] \otimes \c W \into \D W $,
we can identify $\, \c[\vreg] \otimes \c W $ with the cyclic $\D
W$-module $ \D W/J $, where $ J $ is the left ideal of $ \D W $
generated by  $\,
\partial_\xi \in \D W $, $\, \xi \in V \,$. Explicitly, in terms
of generators, $\D W$ acts on $ \c[\vreg] \otimes \c W $ by
\begin{eqnarray}
&& g(f\otimes u) =  gf\otimes u\,,\quad g\in\c[\vreg]\ , \nonumber\\
\la{acts}
&& \partial_\xi(f\otimes u) = \partial_\xi f\otimes u\,,\quad \xi\in V\ ,\\
\nonumber
&& w(f\otimes u)=f^w\otimes wu\,,\quad w\in W\ .
\end{eqnarray}

Now, the restriction of scalars via the Dunkl representation
$\,H_k(W) \into \D W $ makes $ \c[\vreg]\otimes \c W $ an $ H_k(W)
$-module. We will call the corresponding action of $ H_k $ the {\it
differential action}. It turns out that, in the case of integral
$k$'s, the differential action of $ H_k $ is intimately related to
quasi-invariants $Q_k=Q_k(W)$.

Besides the diagonal action \eqref{acts}, we will use another action
of $W$ on $\, \c[\vreg] \otimes \c W $, which is trivial on the
first factor: i.~e., $\, f
\otimes s \mapsto f \otimes w s\,$, where $ w \in W $ and $\, f
\otimes s \in \c[\vreg] \otimes \c W \,$. We denote this action by
$\, 1 \otimes w \,$.

Now, we define $ \QQ_k $ to be the subspace of $\,\c [\vreg] \otimes
\c W $ spanned by the elements $ \varphi $ satisfying
\begin{equation}
\label{qcc} (1 \otimes \e_{H,i})\,\varphi \equiv 0 \ \mbox{mod}\
\langle\alpha_H\rangle^{n_H k_{H,i}} \otimes \c W \ ,
\end{equation}
for all  $\, H \in \A\,$ and $\,i = 0, 1, \ldots, n_H-1 \,$.
Here, as in Definition~\ref{qc}, $\,\langle \alpha_H \rangle\,$
stands for the ideal of $ \c[V] $ generated by  $ \alpha_H $.

\begin{theorem}
\la{Qfat} If $k$ is integral, then  $ \c[\vreg]\otimes \c W  $
contains a unique $H_k$-submodule $ \QQ'_k = \QQ'_k(W) $, such that
$ \QQ'_k $ is finite over $ \c[V] \subset H_k $ and
\begin{equation}
\la{proj} \e \,\QQ'_k = \e (\,Q_k\otimes 1) \quad \mbox{in}\quad
\c[\vreg]\otimes \c W  \ .
\end{equation}
In fact, we have the equality $\QQ_k = \QQ'_k$.
\end{theorem}
As a simple consequence of the theorem, we get the following corollary.
\begin{cor}
\la{sp} $\,Q_k \,$ is stable under the action of $\, U_k \,$ on $
\c[\vreg] $ via the Dunkl representation \eqref{HC}.
Thus $ Q_k $ is a $ U_k$-module, with $ U_k $ acting on
$ Q_k $ by invariant differential operators.
\end{cor}
\begin{proof}
Theorem \ref{Qfat} implies that $\e H_k \e(\e\QQ_k)\subseteq
\e\QQ_k$. Recall that for every element $\e L \e\in\e H_k\e$ we have
$\e L \e=\e\,\Res\,L$, by the definition of the map \eqref{HC}. As a
result,
$$
\e\,(\Res\,L[Q_k]\otimes 1) = \e\,\Res\,L\,[Q_k\otimes 1] = (\e\,
L\, \e)[\QQ_k] \subseteq \e \QQ_k = \e (Q_k\otimes 1)\,.
$$
It follows that $\, (\Res\,L)[Q_k] \subseteq Q_k \,$, since $\,
\e\,(f \otimes 1) = 0\,$ in $\c[\vreg]\otimes \c W$ forces $\,f = 0
\,$.
\end{proof}

\begin{example}\la{exx}
We illustrate Theorem \ref{Qfat} in the one-dimensional case.
Let $\,W=\Z/n\Z \,$ and $\, k = (k_0, \dots, k_{n-1})$ be as in
Example \ref{1dim}. Then
\begin{equation}\la{nc11}
\QQ_k=\bigoplus_{i=0}^{n-1} x^{nk_i}\c[x]\otimes\e_{i}\ ,\quad
\,\e_i=\frac{1}{n}\sum_{w\in W}(\det w)^{-i} w\,.
\end{equation}
Clearly, $ \QQ_k $ is stable under the
action of $W$ and $\c[x]$. On the other hand, if $\, k_i \in \Z\,$,
a short calculation shows that the Dunkl operator $T:=\partial_x-x^{-1}\sum_{i=0}^{n-1} n k_{i} \e_{i}\,$
annihilates the elements $x^{nk_i}\otimes\e_{i}$, and hence preserves $ \QQ_k $ as well.
Now, acting on $ \QQ_k $ by $\e=\e_0$ and using \eqref{exq}, we get
\begin{equation}\la{enc1}
\e\QQ_k=\bigoplus_{i=0}^{n-1}\, x^{nk_i+i}\c[x^n]\otimes\e_{i}= \bigoplus_{i=0}^{n-1}\, 
\e\,(x^{nk_i+i}\c[x^n]\otimes 1) =
\e\,(Q_k\otimes 1)\ ,
\end{equation}
which agrees with Theorem~\ref{Qfat}.
\end{example}

\subsubsection{Differential operators on quasi-invariants}
We briefly recall the definition of differential operators in the
algebro-geometric setting (see \cite{MR01}, Chap.~15).

Let $ A $ be a commutative algebra over $ \c $, and let $ M $ be
an $A$-module. The filtered ring of (linear) differential
operators on $ M $ is defined by
$$
\D_A(M) := \bigcup_{n\ge 0}\,\D_A^n(M) \,\subseteq \,\End_{\c}(M)\
,
$$
where $\, \D^{0}_A(M) := \End_A(M) \,$ and $\, \D^n_A(M) \,$, with
$ n\ge 1 $, are given inductively:
$$
\D_A^n(M) := \{D \in \End_{\c}(M)\ |\ [\,D,\, a\,] \in
\D_A^{n-1}(M) \ \text{for all}\ a \in A\}\ .
$$
The elements of $ \D^n_A(M)\setminus \D^{n-1}_A(M) $ are called
{\it differential operators of order $ n $} on $ M $. Note that
the commutator of two operators in $\D_A^n(M) $ of orders $ n $
and $ m $ has order at most $ n+m-1 $. Hence the associated graded
ring $\,\grd\,\D_A(M) := \bigoplus_{n\ge
0}\D_A^n(M)/\D_A^{n-1}(M)\,$ is a commutative algebra.

If $ X $ is an affine variety with coordinate ring $ A =\O(X) $,
we denote $ \D_A(A) $ by $ \D(X) $ and call it the {\it ring of
differential operators on $ X $}. If $ X $ is irreducible, then
each differential operator on $ X $ has a unique extension to a
differential operator on $ \k := \c(X) $, the field of rational
functions of $ X $, and thus we can identify (see \cite{MR01},
Theorem~15.5.5):
\begin{equation*}
\D(X) = \{\, D \in \D(\k)\ |\ D(f) \in \O(X) \ \text{for all} \ f
\in \O(X) \,\} \ .
\end{equation*}
Slightly more  generally, we have
\begin{lemma}
\la{dm} Suppose that $\, M \subseteq \k $ is a (nonzero)
$A$-submodule of $\, \k $. Then
\begin{equation*}
\D_A(M) = \{\, D \in \D(\k)\ |\ D(f) \in M \ \text{\rm for all} \
f \in M \,\} \ .
\end{equation*}
\end{lemma}

We apply these concepts for $\, A = A_k \,$ and
$\, M = Q_k \,$, writing $ \D(Q_k) $ instead of $ \D_{A}(M) $ in this case.
By Lemma~\ref{alg}(3), $\,X_k = \Spec(A_k) \,$ is an
irreducible variety with $ \k = \c(V) $, so, by Lemma~\ref{dm}, we
have
\begin{equation}
\la{tw} \D(Q_k) = \{D\in \D(\k)\,|\,D(f)\subseteq Q_k\ \text{\rm
for all} \ f \in Q_k \,\}\ .
\end{equation}
Note that the differential filtration on $ \D(Q_k) $ is induced from
the differential filtration on $ \D(\k) $. Thus \eqref{tw} yields
a canonical inclusion $\,\grd\,\D(Q_k) \subseteq \grd\,\D(\k)\,$,
with $\, \D^0(Q_k) = A_k \,$, see \eqref{acc}. In particular, if
$\, k = \{0\} \,$, then $ Q_k = \c[V] $ and \eqref{tw} becomes the
standard realization of $\,\D(V)\,$ as a subring of $ \D(\k)\,$.

Apart from $ Q_k$, we may also apply Lemma~\ref{dm} to $ \c[\vreg]
$, which is naturally a subalgebra of $ \k = \c(V) $. This gives
the identification
\begin{equation}
\la{tw1} \D(\vreg) = \{D\in \D(\k)\,|\,D(f)\subseteq \c[\vreg] \
\text{\rm for all} \ f \in \c[\vreg] \,\} \ .
\end{equation}
\begin{lemma}
\la{symbol} With identifications \eqref{tw} and \eqref{tw1}, we
have
$$
(\mathsf i)\ \D(Q_k)\subseteq \D(\vreg) \quad  \text{and}\quad
(\mathsf{ii})\ \grd\,\D(Q_k)\subseteq \grd\,\D(V)\ .
$$
\end{lemma}

\subsubsection{Invariant differential operators}

Recall that, by Lemma~\ref{alg}, $\, Q_k \,$ is stable under the
action of $W$ on $ \c[\vreg] $. Hence $ W $ acts naturally on $
\D(Q_k) $, and this action is compatible with the inclusion of
Lemma~\ref{symbol}($\mathsf i$). It follows that $\, \D(Q_k)^W
\subseteq \D(\vreg)^W\,$. Now, we recall the algebra embedding \eqref{HC},
which defines the Dunkl representation for the spherical subalgebra
of $ H_k $.
\begin{prop}
\label{is} The image of $\,\Res: \,U_k \into \D(\vreg)^W$  coincides with
$ \D(Q_k)^W $. Thus, the Dunkl representation of $ U_k $ yields an algebra
isomorphism $\,U_k \cong \D(Q_k)^W$.
\end{prop}
Proposition~\ref{is} explains the remark at the end of 
Section~\ref{heckmansargumentsection}.
\begin{cor}\la{grcof}
$\ \grd\,\D(V) \,$ is a finite module over
$\,\grd\,\D(Q_k) \,$. Consequently $\,\grd\,\D(Q_k)\,$
is a finitely generated (and hence, Noetherian) commutative $\c$-algebra.
\end{cor}

We are now in a position to state some of the main results of \cite{BC11}.
The first theorem can be viewed as a generalization of the Chevalley-Serre-Shephard-Todd
Theorem ({\it cf.} Theorem~\ref{sct}, $\,(1)\,\Rightarrow\,(2)\,$).
\begin{theorem}[\cite{BC11}, Theorem~1.1]
\label{mattt}
$ Q_k $ is a free module over $ \c[V]^W \! $ of rank $|W|$.
\end{theorem}
It would be nice to have a direct proof of this theorem
generalizing the homological arguments presented in Section~\ref{frg}. 
Unfortunately, we are not aware of such a generalization. Instead,
Theorem~\ref{mattt} is deduced as a consequence of the following much deeper 
result on the algebra of differential operators on quasi-invariants.
\begin{theorem}[\cite{BC11}, Theorem~1.2]
\label{ma}
$\D(Q_k)$ is a simple ring, Morita equivalent to $\D(V)$.
\end{theorem}
This result is surprising since, as explained above, 
$ \D(Q_k) $ is isomorphic to the ring of (twisted) differential operators
on a {\it singular} algebraic variety, and such rings usually do not
have good properties. 
Combined with standard Morita theory Theorem~\ref{ma} implies that 
$ \P := \{D \in \D(\k)\,:\,D(f) \in Q_k \ \mbox{\rm for all}\ f \in \c[V]\}\,$
is a projective right ideal of $ \D(V)$. This gives one of the only families of examples of non-free projective modules over higher Weyl algebras. In the Coxeter case, the Morita 
equivalence between the algebras $ \D(Q_k)$ for integral $k$'s was originally proved in \cite{BEG03}.

Theorem~\ref{ma} and Proposition~\ref{is} have another interesting consequence established 
in \cite{BC11} using $K$-theoretic arguments.
\begin{theorem}[\cite{BC11}, Corollary~4.6]
\la{Dchev}
$ \D(Q_k) $ is a non-free projective module over $ \D(Q_k)^W $.
\end{theorem}
This result can be viewed as a noncommutative counterpart of Theorem~\ref{mattt}.

\bibliography{secondbibtexfile_3_6}{}

\providecommand{\bysame}{\leavevmode\hbox to3em{\hrulefill}\thinspace}
\providecommand{\MR}{\relax\ifhmode\unskip\space\fi MR }
\providecommand{\MRhref}[2]{%
  \href{http://www.ams.org/mathscinet-getitem?mr=#1}{#2}
}
\providecommand{\href}[2]{#2}
\begin{thebibliography}{GGOR03}

\bibitem[BC11]{BC11}
Yuri Berest and Oleg Chalykh, \emph{Quasi-invariants of complex reflection
  groups}, Compos. Math. \textbf{147} (2011), no.~3, 965--1002. \MR{2801407}

\bibitem[BEG03a]{BEG03}
Yuri Berest, Pavel Etingof, and Victor Ginzburg, \emph{Cherednik algebras and
  differential operators on quasi-invariants}, Duke Math. J. \textbf{118}
  (2003), no.~2, 279--337. \MR{1980996 (2004f:16039)}

\bibitem[BEG03b]{BEG03b}
\bysame, \emph{Finite-dimensional representations of rational {C}herednik
  algebras}, Int. Math. Res. Not. (2003), no.~19, 1053--1088. \MR{1961261
  (2004h:16027)}

\bibitem[BM08]{BM08}
Jason Bandlow and Gregg Musiker, \emph{A new characterization for the
  {$m$}-quasiinvariants of {$S_n$} and explicit basis for two row hook shapes},
  J. Combin. Theory Ser. A \textbf{115} (2008), no.~8, 1333--1357. \MR{2455582
  (2009m:05194)}

\bibitem[BMR98]{BMR98}
Michel Brou{\'e}, Gunter Malle, and Rapha{\"e}l Rouquier, \emph{Complex
  reflection groups, braid groups, {H}ecke algebras}, J. Reine Angew. Math.
  \textbf{500} (1998), 127--190. \MR{1637497 (99m:20088)}

\bibitem[Bou68]{Bou68}
N.~Bourbaki, \emph{\'{E}l\'ements de math\'ematique. {F}asc. {XXXIV}. {G}roupes
  et alg\`ebres de {L}ie. {C}hapitre {IV}: {G}roupes de {C}oxeter et syst\`emes
  de {T}its. {C}hapitre {V}: {G}roupes engendr\'es par des r\'eflexions.
  {C}hapitre {VI}: syst\`emes de racines}, Actualit\'es Scientifiques et
  Industrielles, No. 1337, Hermann, Paris, 1968. \MR{0240238 (39 \#1590)}

\bibitem[CE74]{CE74}
Allan Clark and John Ewing, \emph{The realization of polynomial algebras as
  cohomology rings}, Pacific J. Math. \textbf{50} (1974), 425--434. \MR{0367979
  (51 \#4221)}

\bibitem[Che55]{Che55}
Claude Chevalley, \emph{Invariants of finite groups generated by reflections},
  Amer. J. Math. \textbf{77} (1955), 778--782. \MR{0072877 (17,345d)}

\bibitem[Cox34]{Cox34}
H.~S.~M. Coxeter, \emph{Discrete groups generated by reflections}, Ann. of
  Math. (2) \textbf{35} (1934), no.~3, 588--621. \MR{1503182}

\bibitem[CPS88]{CPS88}
E.~Cline, B.~Parshall, and L.~Scott, \emph{Finite-dimensional algebras and
  highest weight categories}, J. Reine Angew. Math. \textbf{391} (1988),
  85--99. \MR{961165 (90d:18005)}

\bibitem[CV90]{CV90}
O.~A. Chalykh and A.~P. Veselov, \emph{Commutative rings of partial
  differential operators and {L}ie algebras}, Comm. Math. Phys. \textbf{126}
  (1990), no.~3, 597--611. \MR{1032875 (91g:58106)}

\bibitem[CV93]{CV93}
Oleg~A. Chalykh and Alexander~P. Veselov, \emph{Integrability in the theory of
  {S}chr\"odinger operator and harmonic analysis}, Comm. Math. Phys.
  \textbf{152} (1993), no.~1, 29--40. \MR{1207668 (94a:58160)}

\bibitem[DO03]{DO03}
C.~F. Dunkl and E.~M. Opdam, \emph{Dunkl operators for complex reflection
  groups}, Proc. London Math. Soc. (3) \textbf{86} (2003), no.~1, 70--108.
  \MR{1971464 (2004d:20040)}

\bibitem[Dun89]{Dun89}
Charles~F. Dunkl, \emph{Differential-difference operators associated to
  reflection groups}, Trans. Amer. Math. Soc. \textbf{311} (1989), no.~1,
  167--183. \MR{951883 (90k:33027)}

\bibitem[EG02a]{EG02b}
Pavel Etingof and Victor Ginzburg, \emph{On {$m$}-quasi-invariants of a
  {C}oxeter group}, Mosc. Math. J. \textbf{2} (2002), no.~3, 555--566,
  Dedicated to Yuri I. Manin on the occasion of his 65th birthday. \MR{1988972
  (2004g:20052)}

\bibitem[EG02b]{EG02}
\bysame, \emph{Symplectic reflection algebras, {C}alogero-{M}oser space, and
  deformed {H}arish-{C}handra homomorphism}, Invent. Math. \textbf{147} (2002),
  no.~2, 243--348. \MR{1881922 (2003b:16021)}

\bibitem[ES03]{ES03}
Pavel Etingof and Elisabetta Strickland, \emph{Lectures on quasi-invariants of
  {C}oxeter groups and the {C}herednik algebra}, Enseign. Math. (2) \textbf{49}
  (2003), no.~1-2, 35--65. \MR{1998882 (2004h:20052)}

\bibitem[Eti07]{Et07}
Pavel Etingof, \emph{Calogero-{M}oser systems and representation theory},
  Zurich Lectures in Advanced Mathematics, European Mathematical Society (EMS),
  Z\"urich, 2007. \MR{2296754 (2009k:53217)}

\bibitem[FV02]{FV02}
M.~Feigin and A.~P. Veselov, \emph{Quasi-invariants of {C}oxeter groups and
  {$m$}-harmonic polynomials}, Int. Math. Res. Not. (2002), no.~10, 521--545.
  \MR{1883902 (2003j:20067)}

\bibitem[GGOR03]{GGOR03}
Victor Ginzburg, Nicolas Guay, Eric Opdam, and Rapha{\"e}l Rouquier, \emph{On
  the category {$\mathcal O$} for rational {C}herednik algebras}, Invent. Math.
  \textbf{154} (2003), no.~3, 617--651. \MR{2018786 (2005f:20010)}

\bibitem[Gor08]{Gor07}
Iain~G. Gordon, \emph{Symplectic reflection alegebras}, Trends in
  representation theory of algebras and related topics, EMS Ser. Congr. Rep.,
  Eur. Math. Soc., Z\"urich, 2008, pp.~285--347. \MR{2484729 (2009m:16058)}

\bibitem[Gor10]{Gor10}
\bysame, \emph{Rational {C}herednik algebras}, Proceedings of the
  {I}nternational {C}ongress of {M}athematicians. {V}olume {III} (New Delhi),
  Hindustan Book Agency, 2010, pp.~1209--1225. \MR{2827838}

\bibitem[GW04]{GW03}
A.~M. Garsia and N.~R. Wallach, \emph{Some new applications of orbit
  harmonics}, S\'em. Lothar. Combin. \textbf{50} (2003/04), Art. B50j, 47 pp.
  (electronic). \MR{2118048 (2005k:13010)}

\bibitem[GW06]{GW06}
A.~M. Garsia and N.~Wallach, \emph{The non-degeneracy of the bilinear form of
  {$m$}-quasi-invariants}, Adv. in Appl. Math. \textbf{37} (2006), no.~3,
  309--359. \MR{2261177 (2007k:20081)}

\bibitem[Hec91]{Hec91}
Gerrit~J. Heckman, \emph{A remark on the {D}unkl differential-difference
  operators}, Harmonic analysis on reductive groups ({B}runswick, {ME}, 1989),
  Progr. Math., vol. 101, Birkh\"auser Boston, Boston, MA, 1991, pp.~181--191.
  \MR{1168482 (94c:20075)}

\bibitem[Hum78]{Hum78}
James~E. Humphreys, \emph{Introduction to {L}ie algebras and representation
  theory}, Graduate Texts in Mathematics, vol.~9, Springer-Verlag, New York,
  1978, Second printing, revised. \MR{499562 (81b:17007)}

\bibitem[Hum08]{Hum08}
\bysame, \emph{Representations of semisimple {L}ie algebras in the {BGG}
  category {$\mathcal{O}$}}, Graduate Studies in Mathematics, vol.~94, American
  Mathematical Society, Providence, RI, 2008. \MR{2428237 (2009f:17013)}

\bibitem[Jos82]{Jos82}
A.~Joseph, \emph{The {E}nright functor on the {B}ernstein-{G}elfand-{G}elfand
  category {${\mathcal O}$}}, Invent. Math. \textbf{67} (1982), no.~3,
  423--445. \MR{664114 (84j:17005)}

\bibitem[MR01]{MR01}
J.~C. McConnell and J.~C. Robson, \emph{Noncommutative {N}oetherian rings},
  revised ed., Graduate Studies in Mathematics, vol.~30, American Mathematical
  Society, Providence, RI, 2001, With the cooperation of L. W. Small.
  \MR{MR1811901 (2001i:16039)}

\bibitem[OP83]{OP83}
M.~A. Olshanetsky and A.~M. Perelomov, \emph{Quantum integrable systems related
  to {L}ie algebras}, Phys. Rep. \textbf{94} (1983), no.~6, 313--404.
  \MR{708017 (84k:81007)}

\bibitem[Opd95]{Opd95}
E.~M. Opdam, \emph{A remark on the irreducible characters and fake degrees of
  finite real reflection groups}, Invent. Math. \textbf{120} (1995), no.~3,
  447--454. \MR{1334480 (96e:20011)}

\bibitem[Opd00]{Opd00}
Eric~M. Opdam, \emph{Lecture notes on {D}unkl operators for real and complex
  reflection groups}, MSJ Memoirs, vol.~8, Mathematical Society of Japan,
  Tokyo, 2000, With a preface by Toshio Oshima. \MR{1805058 (2003c:33001)}

\bibitem[Rou05]{Rou05}
Rapha{\"e}l Rouquier, \emph{Representations of rational {C}herednik algebras},
  Infinite-dimensional aspects of representation theory and applications,
  Contemp. Math., vol. 392, Amer. Math. Soc., Providence, RI, 2005,
  pp.~103--131. \MR{2189874 (2007d:20006)}

\bibitem[Ser00]{Se00}
Jean-Pierre Serre, \emph{Local algebra}, Springer Monographs in Mathematics,
  Springer-Verlag, Berlin, 2000, Translated from the French by CheeWhye Chin
  and revised by the author. \MR{1771925 (2001b:13001)}

\bibitem[Smi85]{Sm85}
Larry Smith, \emph{On the invariant theory of finite pseudoreflection groups},
  Arch. Math. (Basel) \textbf{44} (1985), no.~3, 225--228. \MR{784089
  (87b:20018)}

\bibitem[ST54]{ST54}
G.~C. Shephard and J.~A. Todd, \emph{Finite unitary reflection groups},
  Canadian J. Math. \textbf{6} (1954), 274--304. \MR{0059914 (15,600b)}

\bibitem[Ves94]{Ve94}
A.~P. Veselov, \emph{Calogero quantum problem, {K}nizhnik-{Z}amolodchikov
  equation and {H}uygens principle}, Teoret. Mat. Fiz. \textbf{98} (1994),
  no.~3, 524--535. \MR{1304748 (96c:58135)}

\bibitem[VSC93]{VSC93}
A.~P. Veselov, K.~L. Styrkas, and O.~A. Chalykh, \emph{Algebraic integrability
  for the {S}chr\"odinger equation, and groups generated by reflections},
  Teoret. Mat. Fiz. \textbf{94} (1993), no.~2, 253--275. \MR{1221735
  (94j:35151)}

\bibitem[Wil98]{Wil98}
George Wilson, \emph{Collisions of {C}alogero-{M}oser particles and an adelic
  {G}rassmannian}, Invent. Math. \textbf{133} (1998), no.~1, 1--41, With an
  appendix by I. G. Macdonald. \MR{1626461 (99f:58107)}

\end{thebibliography}
\bibliographystyle{amsalpha}

\end{document}